\documentclass[12pt,twoside]{amsart}
\usepackage{amssymb,amsmath,amsthm, amscd, enumerate, mathrsfs}
\usepackage[all]{xy}

\title{Relative Bertini type theorem for multiplier ideal sheaves}
\author{Osamu Fujino} 
\date{2018/2/18, version 0.26}
\address{Department of Mathematics, Graduate School of Science, 
Osaka University, Toyonaka, Osaka 560-0043, Japan}
\email{fujino@math.sci.osaka-u.ac.jp}

\DeclareMathOperator{\Supp}{Supp}
\newcommand{\Coker}[0]{\operatorname{Coker}}

\newtheorem{thm}{Theorem}[section]
\newtheorem{lem}[thm]{Lemma}
\newtheorem{cor}[thm]{Corollary}
\newtheorem*{claim}{Claim}
\newtheorem{question}[thm]{Question}

\theoremstyle{definition}
\newtheorem{ex}[thm]{Example}
\newtheorem{defn}[thm]{Definition}
\newtheorem{rem}[thm]{Remark}
\newtheorem*{ack}{Acknowledgments}       

\newtheorem{step}{Step}

\makeatletter
    
    \@addtoreset{equation}{section}
\makeatother
\begin{document}

\maketitle 

\begin{abstract}
We establish a relative Bertini type theorem for multiplier 
ideal sheaves. Then we prove a 
relative version of the Koll\'ar--Nadel type vanishing theorem as 
an application. 
\end{abstract}

\tableofcontents

\section{Introduction}\label{f-sec1}

Let $X$ be a smooth complex projective variety and let $D$ be an effective 
$\mathbb Q$-divisor on $X$. 
Let $H$ be a general member of a very ample linear system 
$\Lambda$ on $X$. 
Then it is well known and is easy to see that 
the equality 
$$
\mathscr J(H, D|_H)=\mathscr J(X, D)|_H
$$ 
holds and that there exists the following short exact sequence 
$$
0\to \mathscr J(X, D)\otimes \mathscr O_X(-H)
\to \mathscr J(X, D)\to \mathscr J(H, D|_H)\to 0, 
$$
where $\mathscr J(X, D)$ (resp.~$\mathscr J(H, D|_H)$) is the 
multiplier ideal sheaf associated to $D$ (resp.~$D|_H$). Let $\varphi$ be a 
quasi-plurisubharmonic function on $X$. 
Then the inclusion 
$$
\mathscr J(\varphi|_H)\subset \mathscr J(\varphi)|_H
$$ 
follows from the Ohsawa--Takegoshi $L^2$ extension theorem. 
Note that $\mathscr J(\varphi)$ (resp.~$\mathscr J(\varphi|_H)$) is 
the multiplier ideal sheaf associated to $\varphi$ (resp.~$\varphi|_H$). 
However, the equality 
$$
\mathscr J(\varphi|_H)=\mathscr J(\varphi)|_H
$$ 
does not always hold. We think that 
the existence of a smooth member $H_0$ of $\Lambda$ such that 
the equality 
$$
\mathscr J(\varphi|_{H_0})=\mathscr J(\varphi)|_{H_0}
$$ 
holds and that there exists the following natural short exact sequence 
$$
0\to \mathscr J(\varphi)\otimes \mathscr O_X(-H_0)
\to \mathscr J(\varphi)\to \mathscr J(\varphi|_{H_0})\to 0
$$ 
is highly nontrivial. In \cite[Theorem 1.10]{fujino-matsumura}, 
we established that there are many members of $\Lambda$ satisfying 
the above good properties. 
The main purpose of this paper is to prove 
the following theorem. 

\begin{thm}[Relative Bertini type theorem for 
multiplier ideal sheaves]\label{f-thm1.1}
Let $f:X\to S$ be a proper surjective morphism 
from a complex manifold $X$ to a complex analytic space $S$. 
Let $\varphi$ be a quasi-plurisubharmonic function on $X$. 
We consider the following commutative diagram: 
$$
\xymatrix{
X\ar[drr]^-h\ar[dd]_-f\ar[dr]_-g &\\ 
&S\times \mathbb P^N \ar[dl]^-{p_1}\ar[r]_-{p_2}& \mathbb P^N\\
S&
}
$$
where $p_i$ is the $i$-th projection for $i=1, 2$. 
We consider the complete linear system 
$$
\Lambda:= |\mathscr O_{\mathbb P^N}(1)|
\simeq \mathbb P^N
$$ on $\mathbb P^N$. 
Let $S^\dag$ be any relatively compact open subset of $S$. 
We put $X^\dag:=f^{-1}(S^\dag)$, $h^\dag:=h|_{X^\dag}$, and consider 
$$
\mathcal G:=
\left\{H'\in \Lambda \, \left|
\begin{array}{l}
\mbox{$H^\dag:=(h^\dag)^*H'$ is well-defined and is a smooth 
divisor on $X^\dag$,}\\ 
\mbox{and 
$\mathscr J(\varphi|_{H^\dag})=\mathscr J (\varphi)|_{H^\dag}$ holds}
\end{array}
\right.
\right\}\subset \Lambda.
$$ 
We note that $H^\dag$ is well-defined if and only if the image of every irreducible 
component of $X^\dag$ by $h^\dag$ is not 
contained in $H'$. 
We also note that $\mathscr J(\varphi)$ {\em{(}}resp.~$\mathscr 
J(\varphi|_{H^\dag})${\em{)}} is the multiplier ideal 
sheaf on $X$ {\em{(}}resp.~$H^\dag${\em{)}} 
associated to $\varphi$ {\em{(}}resp.~$\varphi|_{H^\dag}${\em{)}}. 
Then $\mathcal G$ is dense in $\Lambda$ {\em{(}}$\simeq \mathbb P^N${\em{)}} 
in the classical topology. Furthermore, we put 
$$
\mathcal H:=\{ H'\in \mathcal G\, |\, \text{$H^\dag=(h^\dag)^*H'$ 
contains no associated 
primes of $\mathscr O_X/\mathscr J(\varphi)$ on $X^\dag$}\} \subset \mathcal G. 
$$ 
Then $\mathcal H$ is also dense in $\Lambda$ in the classical topology. 
More generally, $\mathcal G\setminus 
\mathcal S$ and $\mathcal H\setminus \mathcal S$ 
are dense in $\Lambda$ in the classical topology for any analytically meagre subset $\mathcal S$ of 
$\Lambda$.  
We note that there exists the following natural 
short exact sequence 
$$
0\to \mathscr J(\varphi|_{X^\dag})\otimes 
\mathscr O_{X^\dag}(-H^\dag)\to \mathscr J(\varphi|_{X^\dag})
\to \mathscr J(\varphi|_{H^\dag})\to 0
$$ 
for every $H'\in \mathcal H$, where $H^\dag=(h^\dag)^*H'$. 
\end{thm}

In Theorem \ref{f-thm1.1}, we are mainly interested in 
the case where there exists $s_0\in S^\dag$ such that 
$\dim g(f^{-1}(s_0))>0$. 
We want to cut down $g(f^{-1}(s_0))$ by the 
linear system $\Lambda$ in some applications 
(see the proof of Theorem \ref{f-thm1.4} below). 
It may happen that $H^\dag=0$ for some  
$H\in \mathcal G$ when $\dim g(f^{-1}(s))=0$ holds 
for every $s\in S^\dag$. 

If $S=S_0$ is a point, then Theorem \ref{f-thm1.1} is 
essentially the same as \cite[Theorem 1.10]{fujino-matsumura} (see also 
\cite[Corollary 3.13]{fujino-matsumura}). 
Therefore, Theorem \ref{f-thm1.1} can be seen as a 
relative generalization of \cite[Theorem 1.10]{fujino-matsumura}. 
We note that $\Lambda\setminus \mathcal G$ is not always 
analytically meagre in the sense of Definition \ref{f-def3.1} 
(see, for example, \cite[Example 3.12]{fujino-matsumura}). 
The following question is due to S\'ebastien Boucksom in the 
case where $S=S_0$ is a point. 

\begin{question}\label{f-question1.2}
In Theorem \ref{f-thm1.1}, is 
$\Lambda\setminus \mathcal G$ a pluripolar 
subset of $\Lambda$ {\em{(}}$\simeq\mathbb P^N${\em{)}}?
\end{question}

We recommend the reader to 
see \cite[Chapter I.~(5.22) Definition]{demailly-book} 
for the definition of {\em{pluripolar subsets}}.   

Let us recall one of the main results of \cite{matsumura2}, 
which is a relative version of \cite[Theorem A]{fujino-matsumura}. 
Of course, the proof of Theorem \ref{f-thm1.3} in \cite{matsumura2} 
is much harder than that of \cite[Theorem A]{fujino-matsumura}. 
Theorem \ref{f-thm1.3} is a generalization of the Enoki injectivity theorem in 
\cite{enoki}, which is an analytic counterpart of the Koll\'ar injectivity 
theorem (see \cite{kollar}). 

\begin{thm}[{\cite[Theorem 1.3]{matsumura2}}]\label{f-thm1.3}
Let $\pi:X\to S$ be a proper surjective 
locally K\"ahler morphism from a complex manifold $X$ to 
a complex analytic space $S$. 
Let $F$ be a holomorphic line bundle on $X$ equipped with 
a singular hermitian metric $h$ and 
let $M$ be a holomorphic line bundle on $X$ 
with a smooth hermitian metric $h_M$. 
Assume that 
$$\sqrt{-1}\Theta_{h_M}(M)\geq 0 
\quad 
\text{and} \quad
\sqrt{-1}
\Theta_h(F)-\varepsilon \sqrt{-1}\Theta_{h_M}(M)\geq 0
$$ 
for some $\varepsilon>0$. 
Then, for any non-zero holomorphic 
section $s$ of $M$, the map 
$$
\times s: R^q\pi_*(\omega_X\otimes F\otimes \mathscr J(h))\to 
R^q\pi_*(\omega_X\otimes F\otimes \mathscr J(h)\otimes M)
$$ 
induced by the 
tensor product with $s$ 
is injective for every $q$, where $\omega_X$ is the canonical 
bundle of $X$ and $\mathscr J(h)$ is the multiplier ideal 
sheaf associated to the singular hermitian metric $h$. 
\end{thm}

By using Theorems \ref{f-thm1.1} and \ref{f-thm1.3}, 
we prove a relative version of the Koll\'ar--Nadel type vanishing 
theorem (see \cite{fujino2}). 

\begin{thm}[Relative Koll\'ar--Nadel type vanishing theorem]\label{f-thm1.4}
Let $f:X\to Y$ be a proper surjective locally K\"ahler morphism 
from a complex manifold $X$ to a complex analytic 
space $Y$. 
Let $\pi:Y\to Z$ be a projective surjective  
morphism between complex analytic spaces. 
Let $F$ be a holomorphic line bundle on $X$ equipped with 
a singular hermitian metric $h$. 
Let $H$ be a $\pi$-ample holomorphic 
line bundle on $Y$. 
Assume that there exists a smooth hermitian metric $g$ on $f^*H$ such that 
$$
\sqrt{-1}\Theta_g(f^*H)\geq 0 \quad \text{and}\quad 
\sqrt{-1}\Theta_h(F)-\varepsilon 
\sqrt{-1}\Theta_g(f^*H)\geq 0
$$ 
for some $\varepsilon >0$. 
Then we have 
$$
R^i\pi_*R^jf_*(\omega_X\otimes F\otimes \mathscr J(h))=0
$$ 
for every $i>0$ and $j$, 
where $\omega_X$ is the canonical bundle of $X$ 
and $\mathscr J(h)$ is the multiplier ideal 
sheaf associated to the singular hermitian metric $h$. 
\end{thm}

As an application of Theorem \ref{f-thm1.4} and 
the strong openness in \cite{guan-zhou}, we have: 

\begin{cor}\label{f-cor1.5}
Let $f:X\to Y$ be a proper surjective locally K\"ahler morphism 
from a complex manifold $X$ to a complex analytic 
space $Y$. 
Let $\pi:Y\to Z$ be a locally projective surjective  
morphism between complex analytic spaces. 
Let $F$ be a holomorphic line bundle on $X$ equipped with 
a singular hermitian metric $h$ such that $\sqrt{-1}\Theta_h(F)\geq 0$. 
Let $M$ be a $\pi$-nef and $\pi$-big holomorphic 
line bundle on $Y$. Then we have 
$$
R^i\pi_*(M\otimes R^jf_*(\omega_X\otimes F\otimes \mathscr J(h)))=0
$$ 
for every $i>0$ and $j$, 
where $\omega_X$ is the canonical bundle of $X$ 
and $\mathscr J(h)$ is the multiplier ideal 
sheaf associated to the singular hermitian metric $h$. 
\end{cor}

For related vanishing theorems, see \cite{fujino1}, 
\cite{fujino2}, \cite{kollar}, \cite{matsumura1}, \cite{ohsawa}, \cite{takegoshi}, 
and so on. We recommend the reader to see 
\cite[Chapters 5 and 6]{fujino3}, where 
we discuss various Koll\'ar type vanishing theorems 
by using the theory of mixed Hodge structures 
on cohomology with compact support and 
explain their applications to the minimal model 
program for higher-dimensional complex algebraic varieties. 

We close this introduction with a remark on Nakano semipositive 
vector bundles. 

\begin{rem}[Twists by Nakano semipositive vector bundles]\label{f-rem1.6} 
Let $E$ be a Nakano semipositive holomorphic vector 
bundle on $X$. 
Then it is not difficult to see that 
Theorems \ref{f-thm1.3}, \ref{f-thm1.4}, and Corollary \ref{f-cor1.5} 
hold even when $\omega_X$ is replaced by $\omega_X\otimes E$. 
We leave the details as an exercise for the reader 
(see \cite[Section 6]{fujino-matsumura}). 
\end{rem}
\begin{ack}
The author was partially supported 
by JSPS KAKENHI Grant Numbers JP16H03925, JP16H06337. 
He would like to thank Shin-ichi Matsumura 
very much for some useful comments. 
He also would like to thank 
Professor S\'ebastien Boucksom for comments on \cite{fujino-matsumura}.
\end{ack}

\section{Preliminaries}\label{f-sec2}

For the basic results of the theory of complex analytic spaces, 
see \cite{banica-s} and \cite{fischer}. 
For various analytic methods used in this paper, see \cite{demailly}. 

\begin{defn}[Singular hermitian metrics and curvatures]\label{f-def2.1} 
Let $F$ be a holomorphic line bundle
on a complex manifold $X$. 
A {\em{singular hermitian metric}}
on $F$ is a metric $h$ which is given in 
every trivialization $\theta: F|_U\simeq U\times 
\mathbb C$ 
by 
\begin{equation*}
| \xi |_{h} =|\theta(\xi)|e^{-\varphi} \text{ on } U, 
\end{equation*}
where $\xi$ is a section of $F$ on $U$ and 
$\varphi \in L^1_{\mathrm{loc}}(U)$ 
is an arbitrary function. 
Here $L^1_{\mathrm{loc}}(U)$ is the space
of locally integrable functions on $U$. 
We usually call $\varphi$ 
the weight function of the
metric with respect to the trivialization $\theta$. 
The curvature of a singular hermitian metric $h$ is 
defined by 
\begin{equation*}
\Theta_{h}(F):=2\partial \overline\partial \varphi, 
\end{equation*}
where $\varphi$ is a weight function and 
$\partial \overline\partial\varphi$ is taken in the sense of currents. 
It is easy to see that the right hand side does not depend on 
the choice of trivializations. 
Therefore, we get a global closed $(1, 1)$-current $\Theta_h(F)$ on $X$. 
\end{defn}

\begin{defn}[(Quasi-)plurisubharmonic functions and multiplier ideal 
sheaves]\label{f-def2.2}
A function $\varphi:U\to [-\infty, \infty)$ defined on an 
open set $U\subset \mathbb C^n$ is called 
{\em{plurisubharmonic}} if 
\begin{itemize}
\item[(i)] $\varphi$ is upper semicontinuous, and 
\item[(ii)] for every complex line $L\subset \mathbb C^n$, 
$\varphi|_{U\cap L}$ is subharmonic on $U\cap L$, 
that is, for every $a\in U$ and $\xi \in \mathbb C^n$ satisfying 
$|\xi|<d(a, U^c)=\inf \{|a-x| \, | \, x\in U^c\}$, the function 
$\varphi$ satisfies the mean inequality 
$$
\varphi(a)\leq \frac{1}{2\pi}\int ^{2\pi}_0 \varphi(a+e^{i\theta}
\xi)d\theta. 
$$ 
\end{itemize}

Let $X$ be an $n$-dimensional complex manifold. 
A function $\varphi:X\to [-\infty, \infty)$ is said to 
be {\em{plurisubharmonic}} 
if there exists an open cover $X=\bigcup _{i\in I}U_i$ such that 
$\varphi|_{U_i}$ is plurisubharmonic on $U_i$ ($\subset \mathbb C^n$) for every 
$i$. 
A {\em{quasi-plurisubharmonic}} function is a  function $\varphi$ which 
is locally equal to the sum of a plurisubharmonic function and of a smooth 
function. 

Let $\varphi$ be a quasi-plurisubharmonic function on a complex 
manifold $X$. 
Then the {\em{multiplier ideal sheaf}} 
$\mathscr J(\varphi)\subset \mathscr O_X$ is defined by 
$$
\Gamma (U, \mathscr J (\varphi))=\{ f\in 
\mathscr O_X(U) \, |\,  |f|^2e^{-2\varphi}\in L^1_{\mathrm{loc}}(U)\}
$$ 
for every open set $U\subset X$. 
It is well known that $\mathscr J(\varphi)$ is a coherent ideal sheaf on $X$. 

Let $S$ be a complex submanifold 
of $X$. Then the restriction $\mathscr J(\varphi)|_S$ 
of the multiplier ideal sheaf $\mathscr J(\varphi)$ to $S$ is 
defined by the image of $\mathscr J(\varphi)$ under 
the natural surjective 
morphism $\mathscr O_X\to \mathscr O_S$, 
that is, 
\begin{equation*}
\mathscr J(\varphi)|_S=\mathscr J(\varphi) / \mathscr J(\varphi)
\cap\mathcal I_S, 
\end{equation*} 
where $\mathcal I_S$ is the defining ideal sheaf of 
$S$ on $X$. 
We note that the restriction $\mathscr J(\varphi)|_S$ 
does not always coincide with $\mathscr J(\varphi)\otimes 
\mathscr O_S=\mathscr J(\varphi)/\mathscr J(\varphi)\, \mathcal I_S$. 
\end{defn}

\begin{defn}[Multiplier ideal sheaves associated to singular hermitian 
metrics]\label{f-def2.3}
Let $F$ be a holomorphic line bundle on a complex manifold $X$ and 
let $h$ be a singular hermitian metric on $F$.  
We assume $\sqrt{-1}\Theta_h(F)\geq \gamma$ 
for some smooth $(1, 1)$-form $\gamma$ on $X$.
We fix a smooth hermitian metric $h_{\infty}$ on $F$. 
Then we can write $h=h_{\infty} e^{-2\psi}$ for some 
$\psi \in L^1_{\mathrm{loc}}(X)$.  
Then 
$\psi$ coincides with a quasi-plurisubharmonic 
function $\varphi$ on $X$ almost everywhere. 
In this situation, we put $\mathscr J(h):=\mathscr J(\varphi)$. 
We note that $\mathscr J(h)$ is independent of 
$h_{\infty}$ and is well-defined. 
\end{defn}

\section{Bertini type theorem revisited}\label{f-sec3}

In this section, we will reformulate some results in \cite{manaresi} 
for our purposes. Let us recall the definition of 
{\em{analytically meagre}} subsets. 

\begin{defn}\label{f-def3.1}
A subset $\mathcal S$ of a complex analytic space $X$ 
is said to be {\em{analytically meagre}} 
if $$\mathcal S\subset \bigcup _{n\in \mathbb N} Y_n$$ where 
each $Y_n$ is a locally closed analytic 
subset of $X$ of codimension $\geq 1$. 
\end{defn}

The following result is a slight reformulation of 
\cite[(II.5) Theorem and (II.7) Corollary]{manaresi}. We need it for the proof of 
Theorem \ref{f-thm1.1} in Section \ref{f-sec4}. 

\begin{thm}[Bertini type theorem for complex manifolds]\label{f-thm3.2}
Let $M$ be a complex manifold 
which has a countable base of open subsets 
and let $\mathscr L$ be a holomorphic line 
bundle on $M$. 
Assume that $M$ has only finitely many connected components. 
Let $t_l$ be an element of 
$H^0(M, \mathscr L)$ for 
every $1\leq l\leq N+1$ such that $\{t_1, \ldots, t_{N+1}\}$ 
generates $\mathscr L$, 
that is, $W\otimes _{\mathbb C} \mathscr O_M\to \mathscr L$ is surjective, 
where $W$ is the linear subspace of $H^0(M, \mathscr L)$ spanned 
by $\{t_1, \ldots, t_{N+1}\}$. 
We consider an $(N+1)$-dimensional vector space 
$V=\bigoplus _{l=1}^{N+1} \mathbb C t_l$. 
Then there exists a dense subset $\mathscr D$ of 
$\Lambda =(V-\{0\})/{\mathbb C^{\times}}(\simeq \mathbb P^{N})$ such that 
$\Lambda \setminus \mathscr D$ is analytically meagre and that 
for each element of $\mathscr D$ the corresponding divisor 
on $M$ is smooth. 
\end{thm}

In Theorem \ref{f-thm3.2}, we do not assume that $\{t_1, \ldots, t_{N+1}\}$ is 
linearly independent. 

\begin{proof}[Proof of Theorem \ref{f-thm3.2}]
If $N=0$, then the statement is trivial. Therefore, 
we may assume that $N\geq 1$. 

\begin{step}\label{f-thm3.2-step1} 
In this step, we will prove that there exists a dense subset 
$\mathscr E$ of $V$, which 
is a countable intersection of dense open subsets of $V$, 
such that for every $s\in V$ the zero set $(s=0)$ is a smooth 
divisor on $M$ if and only if $s\in \mathscr E$. 

We take a countable covering $\{K_i\}_{i\in \mathbb N}$ of 
$M$ such that $K_i$ is compact for every $i$. 
We may assume that 
$K_i$ is contained in an open subset $U_i$ of $M$ such that 
there exists $s_i\in V$ which is never zero 
on $U_i$ for every $i$. 
We put 
$$
\mathscr E_i:=
\left\{s\in V\, \left|\, 
\begin{array}{l}
\mbox{$(s=0)$ contains no irreducible components of $M$}\\ 
\mbox{and is smooth at every point of $K_i\cap (s=0)$}
\end{array}
\right.
\right\}.
$$
Then $\mathscr E_i$ is open in $V$ by \cite[I Step in the proof of 
(II.5) Theorem]{manaresi} and is dense in $V$ by 
\cite[II Step in the proof of (II.5) Theorem]{manaresi}. 
We put $\mathscr E=\bigcap _{i\in N}\mathscr E_i$. 
Then $\mathscr E$ is dense in $V$ by the Baire category theorem. 
By definition, for every $s\in V$, $(s=0)$ is a smooth 
divisor on $M$ if and only if 
$s\in \mathscr E$. 
By definition, $\mathscr E_i$ is $\mathbb C^\times$-invariant and  
$\mathscr E_i\subset 
V-\{0\}$. 
We put $p(\mathscr E)=\mathscr D$, where $p: V-\{0\}\to \Lambda$ is the 
natural projection. Of course, $\mathscr D$ is dense in $\Lambda$. 
\end{step}
\begin{step}\label{f-thm3.2-step2}
In this step, we will prove that $\Lambda\setminus \mathscr D$ is a 
countable union of locally closed analytic subsets of $\Lambda$. 

Let $\{U_i\}_{i\in \mathbb N}$ be an open covering 
of $M$ on which $\mathscr L$ is trivial as in Step \ref{f-thm3.2-step1}. 
With respect to this trivialization of $\mathscr L$, 
we can see that every $s\in V$ is a holomorphic function on each 
$U_i$. 
Since the number of connected components of $M$ is finite, 
we can take a finite number of 
lineaer subspaces $\{V_j\}_{j=1}^k$ of $V$ such that 
$V_j\subsetneq V$ for every $j$ and 
that $s\in V$ is not identically zero on 
any irreducible component of $M$ if and only if 
$s\in V\setminus \bigcup _{j=1}^k V_j$. 
For each $i$, we can consider the holomorphic map 
$$
F_i: U_i\times V\to \mathbb C\times V
$$
defined by $F_i(x, s)=(s(x), s)$. 
Since every $s\in V^\dag:=V\setminus \bigcup _{j=1}^k V_j$ is not identically 
zero on any irreducible component of $M$, 
$F_i$ is flat on $U_i\times V^\dag$ 
(see \cite[(II.1) Lemma]{manaresi}). 
We consider 
\begin{equation*}
\begin{split}
A_i&:=\left\{\left.(x, s)\in U_i\times V^\dag 
\, \right|\, 
\text{$s(x)=0$ and $(s=0)$ is not smooth at $x$} 
\right\}\\ 
&=F^{-1}_i(\{0\}\times V^\dag)
\cap 
\{(x, s)\in U_i\times V^\dag\, |\, 
F^{-1}_i(F_i(x,s))\  \text{is not smooth at 
$(x,s)$}\}.
\end{split}
\end{equation*}
Then, by \cite[(0.3) a) Proposition]{manaresi}, $A_i$ is an analytic subset of 
$U_i\times V^\dag$ for every $i$. 
Therefore, 
$$
A:=\bigcup _{i\in \mathbb N} A_i\cup 
\left(M\times \bigcup _{j=1}^k V_j\right)
$$
is a countable union of locally closed analytic subsets of 
$M\times V$. 
By construction, $V\setminus \mathscr E=q(A)$, 
where $q:M\times V\to V$ is the natural projection. 
Therefore, $V\setminus \mathscr E$ is a countable union of locally closed 
analytic subsets by \cite[Lemma in (0.2)]{manaresi}. 
Thus we see that $\Lambda\setminus \mathscr D=p(q(A)-\{0\})$ 
is also a countable union of locally closed analytic subsets 
by \cite[Lemma in (0.2)]{manaresi}. 
\end{step}
Anyway, $\Lambda \setminus \mathscr D$ is analytically meagre since 
$\Lambda\setminus \mathscr D$ is a countable union of 
locally closed analytic subsets by Step \ref{f-thm3.2-step2} and 
$\mathscr D$ is dense by Step \ref{f-thm3.2-step1}.
\end{proof}

Although Theorem \ref{f-thm3.2} is sufficient for 
the proof of Theorem \ref{f-thm1.1} in Section \ref{f-sec4}, 
we add some remarks for the reader's convenience. 

\begin{rem}\label{f-rem3.3}
The proof of Theorem \ref{f-thm3.2} says that 
we can take $\mathscr D$ such that 
$\Lambda \setminus \mathscr D$ is a countable 
union of locally closed analytic subsets of $\Lambda$ of codimension 
$\geq 1$ and that 
for every $s\in \Lambda$ the zero set $(s=0)$ defines 
a smooth divisor on $M$ if and only if $s\in \mathscr D$. 
\end{rem}

\begin{rem}\label{f-rem3.4}
Theorem \ref{f-thm3.2} and Remark \ref{f-rem3.3} hold true without 
assuming that $M$ has only finitely many connected components. 
We assume that $M$ has infinitely many connected components. 
Then we have the following irreducible decomposition 
$M=\bigcup_{n\in \mathbb N}M_n$ since $M$ has a countable 
base of open subsets. 
By applying Theorem \ref{f-thm3.2} and Remark \ref{f-rem3.3} 
to each $M_n$, 
we get a dense subset $\mathscr D_n$ of $\Lambda$ 
with the desired properties for every $n$. 
We put $\mathscr D=\bigcap _{n\in \mathbb N}\mathscr D_n$. 
Then $\Lambda\setminus \mathscr D$ is a countable union of 
locally closed analytic subsets of $\Lambda$ of codimension $\geq 1$, 
and for every $s\in \Lambda$ the zero set $(s=0)$ defines 
a smooth divisor on $M$ if and only if $s\in \mathscr D$. 
\end{rem}

We prepare easy lemmas for the proof of Theorem \ref{f-thm1.1} 
in Section \ref{f-sec4}. 

\begin{lem}\label{f-lem3.5}
Let $\mathcal S$ be an analytically meagre subset of $\mathbb P^N$. 
Let $p:\mathbb P^N-\{P\} \to \mathbb P^{N-1}$ be the linear 
projection from $P\in \mathbb P^N$. 
Then there exists an analytically meagre subset $\mathcal S'$ 
of $\mathbb P^{N-1}$ such that 
$p^{-1}(x)\cap \mathcal S$ is an analytically 
meagre subset of $p^{-1}(x)\simeq 
\mathbb C$ for every $x\in \mathbb P^{N-1}\setminus \mathcal S'$. 
\end{lem}

\begin{proof} 
We may assume that $\mathcal S$ is a countable 
union of locally closed analytic 
subsets of $\mathbb P^N$. 
We note that $p(V-\{P\})$ is a countable union 
of locally closed analytic subsets of $\mathbb P^{N-1}$, 
where $V$ is any locally closed analytic subset of $\mathbb P^N$ 
(see, for example, \cite[Lemma in (0.2)]{manaresi}). 
By taking a suitable subdivision of $\mathcal S$ into 
locally closed analytic subsets of $\mathbb P^N$, 
we can write 
$$
\mathcal S=\left(\bigcup _{j\in \mathbb N} Y_j\right)
\cup 
\left(\bigcup _{k\in \mathbb N}Z_k\right), 
$$ 
where $\dim Y_j=N-1$ such that $p: Y_j-\{P\}\to \mathbb P^{N-1}$ 
has no positive dimensional fibers for every $j$, and 
any irreducible component of $p(Z_k-\{P\})$ has dimension 
$\leq N-2$ for every $k$. 
We put $\mathcal S'=\bigcup _{k\in \mathbb N}p(Z_k-\{P\})$. 
Then $\mathcal S'$ satisfies the desired properties. 
\end{proof}

Lemma \ref{f-lem3.6} will play an important role 
in the induction on $N$. 

\begin{lem}\label{f-lem3.6}
Let $\mathcal G_N$ be a subset of $\mathbb P^N$ and 
let $\Sigma$ be an analytically meagre subset of $\mathbb P^N$. 
Let $\mathcal G_{N-1}$ be a subset of $\mathbb P^{N-1}$ such that 
$\mathcal G_{N-1}\setminus \mathcal S_{N-1}$ is dense 
in $\mathbb P^{N-1}$ in the classical topology 
for any analytically meagre subset $\mathcal 
S_{N-1}$ of $\mathbb P^{N-1}$. 
Let $p:\mathbb P^N-\{P\}\to \mathbb P^{N-1}$ be the 
linear projection from $P\in \mathbb P^N$. 
Assume that almost all points of $p^{-1}(x)$ is contained 
in $\mathcal G_N$ for every $x\in \mathcal G_{N-1}$ with 
$p^{-1}(x)\setminus \Sigma\ne \emptyset$. Then 
$\mathcal G_N\setminus \mathcal S_N$ is dense 
in $\mathbb P^N$ in 
the classical topology for any analytically meagre subset $\mathcal S_N$ 
of $\mathbb P^N$. 
\end{lem}

\begin{proof}
We put $\mathcal S=\Sigma\cup \mathcal S_N$. 
Then $\mathcal S$ is an analytically meagre subset of $\mathbb P^N$. 
We can define an analytically meagre subset $\mathcal S'$ of $\mathbb P^{N-1}$ 
as in the proof of Lemma \ref{f-lem3.5}. 
Then $\mathcal G_{N-1}\setminus \mathcal S'$ is dense in 
$\mathbb P^{N-1}$ in the classical topology by assumption. 
By assumption again, almost all points of $p^{-1}(x)$ is contained in 
$\mathcal G_N\setminus \mathcal S_N$ for every $x\in \mathcal G_{N-1}\setminus 
\mathcal S'$. 
Therefore, we can easily see that $\mathcal G_N\setminus 
\mathcal S_N$ is dense in $\mathbb P^N$ in the classical topology. 
\end{proof}

We will use Lemma \ref{f-lem3.6} in order to prove the density of $\mathcal G$ in 
Theorem \ref{f-thm1.1} by induction on $N$. 

\begin{rem}\label{f-rem3.7}
In Lemma \ref{f-lem3.6}, we assume that 
$\mathbb P^N$ is the linear system 
$\Lambda=|\mathscr O_{\mathbb P^N}(1)|$ as  
in Theorem \ref{f-thm1.1}. 
We assume that $P\in \mathbb P^N=\Lambda$ corresponds to 
a hyperplane $H'_0$ on the original projective space $\mathbb P^N$. 
Let $p:\mathbb P^N-\{P\}\to \mathbb P^{N-1}$ be the linear 
projection as in Lemma \ref{f-lem3.6}. 
Then we can see $\mathbb P^{N-1}$ as the 
linear system $\Lambda|_{H'_0}$. 
\end{rem}

\section{Proof of Theorem \ref{f-thm1.1}}\label{f-sec4}

In this section, we will prove Theorem \ref{f-thm1.1}. 
We prepare some lemmas before we start the 
proof of Theorem \ref{f-thm1.1}. 
Lemma \ref{f-lem4.1} is nothing but \cite[Lemma 3.2]{fujino-matsumura}. 
Note that a main ingredient of 
Lemma \ref{f-lem4.1} is the Ohsawa--Takegoshi $L^2$ extension 
theorem. 

\begin{lem}\label{f-lem4.1} 
Let $X$ and $\varphi$ be as in Theorem \ref{f-thm1.1}. 
Let $H'_i$ be an element of $\Lambda$ for $1\leq i\leq k$. 
We assume the following condition:  
\begin{itemize}
\item[$\spadesuit$]\label{cond-spade} 
$H^\dag_i:=
(h^\dag)^*H'_i$ is a well-defined smooth divisor on $X^\dag$ 
for every $1\leq i\leq k$ and  
$\sum _{i=1}^kH^\dag_i$ is a simple normal crossing 
divisor on $X^\dag$. 
Moreover, 
for every $1\leq i_1<i_2<\cdots < i_l\leq k$ and any $P\in H^\dag_{i_1}
\cap H^\dag_{i_2}\cap \cdots \cap H^\dag_{i_l}$, 
the set $\{f_{i_1}, f_{i_2}, \ldots, f_{i_l}\}$ is a regular sequence 
for $\mathscr O_{X, P}/\mathscr J(\varphi)_P$, where 
$f_i$ is a $($local$)$ defining equation of $H^\dag_i$ for every $i$. 
\end{itemize}
We put $F_i:=H^\dag_1\cap H^\dag_2\cap \cdots \cap H^\dag_i$ 
for $1 \leq i \leq k$. 
We assume that the equality 
$$\mathscr J(\varphi|_{F_k})=\mathscr J(\varphi)|_{F_k}$$ 
holds. 
Then $$\mathscr J(\varphi|_{F_j})=\mathscr J(\varphi)|_{F_j}$$ holds 
in a neighborhood of $F_k$ for every $j$. 
\end{lem}

For the proof and 
the details of Lemma \ref{f-lem4.1}, 
see \cite[Lemmas 3.1 and 3.2, Remark 3.4, and Lemma 3.5]{fujino-matsumura}. 

\begin{rem}\label{f-rem4.2}
Condition $\spadesuit$ in Lemma \ref{f-lem4.1} 
does not depend on the order of $\{H'_1, H'_2, \ldots, H'_k\}$ 
(see \cite[Remark 3.3]{fujino-matsumura}). 
\end{rem}

Lemma \ref{f-lem4.3} 
below is similar to \cite[Lemma 3.6]{fujino-matsumura}. 

\begin{lem}\label{f-lem4.3}
Let $X$ and $\Lambda$ be as in Theorem \ref{f-thm1.1}. 
Let $\Lambda_0$ be an $m$-dimensional 
sublinear system of $\Lambda$ spanned by 
$\{H'_1, \ldots, H'_m, H'_{m+1}\}$ such that 
$\{H'_1, \ldots, H'_m, H'_{m+1}\}$ satisfies $\spadesuit$. 
We put 
$$
\mathscr F_0=\{ H'\in \Lambda_0 \, | \, 
\{H'_1, \ldots, H'_m, H'\} \ \text{satisfies $\spadesuit$}\}. 
$$
Then $\Lambda_0\setminus \mathscr F_0$ is analytically meagre. 

Moreover, we assume that 
$\mathscr J(\varphi|_F)=\mathscr J(\varphi)|_F$ holds, 
where $F$ is an irreducible component of $H^\dag_1\cap 
\cdots \cap H^\dag_{m+1}$. 
Let $H'$ be a member of $\mathscr F_0$. 
Then 
$$
\mathscr J(\varphi|_{H^\dag})=\mathscr J(\varphi)|_{H^\dag}
$$ 
holds in a neighborhood of $F$, where $H^\dag=(h^\dag)^*H'$. 
\end{lem}
\begin{proof}
Let $\widetilde\Lambda_0$ be the sublinear system of $\Lambda_0$ spanned 
by $\{H'_1, \ldots, H'_m\}$. Then we see that 
$$
H^\dag_1\cap \cdots \cap H^\dag_m\cap H^\dag_{m+1}
=H^\dag_1\cap \cdots \cap H^\dag_m\cap H^\dag
$$ holds for every $H'\in \Lambda_0\setminus \widetilde\Lambda_0$. 
We note that the number of irreducible components of 
$H^\dag_{i_1}\cap H^\dag_{i_2}\cap \cdots \cap H^\dag_{i_l}$ with 
$1\leq i_1<i_2<\cdots <i_l\leq m$ is finite. 
We also note that the number of the associated primes of 
$
\mathcal O_{X^\dag}/\mathcal J (\varphi)|_{X^\dag}$ and 
the number of the associated primes of 
$$
\mathcal O_{H^\dag_{i_1}\cap \cdots \cap H^\dag_{i_l}}/
\mathscr J (\varphi) |_{H^\dag_{i_1}\cap \cdots \cap H^\dag_{i_l}}
$$ 
with 
$1\leq i_1<i_2<\cdots <i_l\leq m$ 
are finite (see, for example, \cite[(I.6) Lemma]{manaresi}). 
It is obvious that $H^\dag_1\cap \cdots \cap H^\dag_{m+1}$ is 
empty on $X^\dag\setminus 
H^\dag_1\cap \cdots \cap H^\dag_{m+1}$.  
Therefore, by applying Theorem \ref{f-thm3.2} to 
$X^\dag\setminus H^\dag_1\cap \cdots \cap H^\dag_{m+1}$ and 
${H^\dag_{i_1}\cap \cdots \cap H^\dag_{i_l}}\setminus 
H^\dag_1\cap \cdots \cap H^\dag_{m+1}$ for every 
$1\leq i_1<i_2<\cdots< i_l\leq m$, we can easily check that 
$\Lambda_0\setminus \mathscr F_0$ is analytically meagre. 

Let $H'$ be a member of $\mathscr F_0$. Then 
$$
H^\dag_1\cap \cdots \cap H^\dag_m\cap H^\dag_{m+1}
=H^\dag_1\cap \cdots \cap H^\dag_m\cap H^\dag
$$ 
always holds. Therefore, $F$ is an irreducible component of 
$H^\dag_1\cap \cdots \cap H^\dag_m\cap H^\dag$. Thus, 
by Lemma \ref{f-lem4.1} and Remark \ref{f-rem4.2}, 
the equality 
$\mathscr J(\varphi|_{H^\dag})=\mathscr J(\varphi)|_{H^\dag}$ holds 
in a neighborhood 
of $F$ for every $H'\in \mathcal F_0$. 
\end{proof}

The following example may help the reader understand 
Theorem \ref{f-thm1.1} and its proof given below. 

\begin{ex}\label{f-ex4.4}
We put 
$$
\Delta^n=\{(z_1, \ldots, z_n)\in \mathbb C^n\, |\, |z_1|<1, \cdots, |z_n|<1\}. 
$$ 
Let $\pi:\Delta^n\to \Delta=\{z\in \mathbb C \, |\, |z|<1\}
$ be the projection given by $(z_1, \ldots, z_n)\mapsto z_n$. 
Let $\varphi$ be a quasi-plurisubharmonic function in a 
neighborhood of $\overline {\Delta^n}$, that is, 
the closure of $\Delta^n$ in $\mathbb C^n$. 
Then, by the Ohsawa--Takegoshi $L^2$ extension theorem, 
we have the following inclusion 
$$
\mathscr J(\varphi|_{H_s})\subset \mathscr J(\varphi)|_{H_s}
$$ 
for every $s\in \Delta$, where $H_s=\pi^{-1}(s)$. 
On the other hand, Fubini's theorem implies 
$$
\mathscr J(\varphi|_{H_s})\supset \mathscr J(\varphi)|_{H_s}
$$ 
for almost all $s\in \Delta$. 
Therefore, the equality 
$$
\mathscr J(\varphi|_{H_s})=\mathscr J(\varphi)|_{H_s}
$$ 
holds for almost all $s\in \Delta$. 
We do not know whether 
$$
\{ s\in \Delta\, |\, \mathscr J (\varphi|_{H_s})\subsetneq \mathscr J
(\varphi)|_{H_s}\}
$$ 
is a pluripolar subset of $\Delta$ or not 
(see Question \ref{f-question1.2}). 
\end{ex} 

Let us prove Theorem \ref{f-thm1.1}. 

\begin{proof}[Proof of Theorem \ref{f-thm1.1}]
Without loss of generality, we may assume that $S$ has a 
countable base of open subsets by shrinking $S$ suitably. 
Moreover, by replacing 
$S$ with its smaller relatively compact open subset if necessary, 
we may further assume that 
$S$ is a relatively compact open subset of a complex analytic 
space throughout the proof of Theorem \ref{f-thm1.1}. 
Of course, we may assume that 
every irreducible component of $X$ intersects with 
$X^\dag$ by abandoning unnecessary irreducible components of $X$. 
\setcounter{step}{0}
\begin{step}\label{f-step1}
In this step, we will prove that $\mathcal G$ is dense in $\Lambda$ 
in the classical topology under the assumption that $N=1$. 
More generally, we will see that $\mathcal H$, 
$\mathcal G\setminus \mathcal S$, and $\mathcal H\setminus 
\mathcal S$ are dense in $\Lambda$ in the classical topology 
for any analytically meagre subset $\mathcal S$ of $\Lambda$ under 
the assumption that $N=1$. 

By Sard's theorem (see, for 
example, \cite[(I.1) Theorem]
{manaresi}), there exists a countable subset $\Sigma$ of 
$\mathbb P^1$ such that 
$X_x=h^*x$ is a smooth divisor on $X$ 
for every $x\in \mathbb P^1\setminus 
\Sigma$. Of course, it may happen that 
$h^{-1}(x)$ is empty. 
By the Ohsawa--Takegoshi $L^2$ extension 
theorem, 
we have $$\mathscr J(\varphi|_{X_x})\subset 
\mathscr J(\varphi)|_{X_x}$$ for every $x\in \mathbb P^1\setminus 
\Sigma$. 
On the other hand, by Fubini's theorem, 
we see that  
$$
\mathscr J(\varphi|_{X^\dag_x})\supset \mathscr J(\varphi)|_{X_x^\dag}
$$ 
holds for almost all $x\in \mathbb P^1\setminus \Sigma$, 
where $X_x^\dag:=X_x\cap X^\dag$. 
This means that $\mathcal G$ is dense in $\Lambda \simeq 
\mathbb P^1$ in the classical topology. 
Since there are 
only finitely many associated primes of $\mathscr O_X/\mathscr J(\varphi)$ on 
$X^\dag$ (see, for example, \cite[(I.6) Lemma]{manaresi}), 
$\mathcal G\setminus \mathcal H$ is an analytically meagre subset of 
$\Lambda$. 
We note that $(\Lambda\setminus \mathcal G)\cup 
\mathcal S$ has measure zero for any analytically meagre subset 
$\mathcal S$ of $\Lambda\simeq \mathbb P^1$. 
Therefore, we see that $\mathcal H$, $\mathcal G 
\setminus \mathcal S$, and $\mathcal H\setminus \mathcal S$ 
are dense in $\Lambda$ in the classical topology for any analytically meagre subset 
$\mathcal S$ of $\Lambda$. 
\end{step}
\begin{step}\label{f-step2}
By Step \ref{f-step1}, 
we can 
prove the following lemma. 

\begin{lem}\label{f-lem4.5}
Let $H'_1$ and $H'_2$ be two members of $\Lambda$ 
such that $\{H'_1, H'_2\}$ satisfies $\spadesuit$. 
Let $\mathcal P$ be the pencil spanned by 
$H'_1$ and $H'_2$, that is, $\mathcal P$ is the sublinear 
system of $\Lambda$ spanned by $H'_1$ and $H'_2$. 
Then, for almost all $H'\in \mathcal P$, 
$\{H'\}$ satisfies $\spadesuit$, 
and $$\mathscr J(\varphi|_{H^\dag})=\mathscr 
J(\varphi)|_{H^\dag}$$ holds outside $H^\dag_1\cap H^\dag_2$, 
where $H^\dag=(h^\dag)^*H'$, $H_1^\dag=(h^\dag)^*H'_1$, 
and  $H_2^\dag=(h^\dag)^*H'_2$. 
\end{lem}

\begin{proof}[Proof of Lemma \ref{f-lem4.5}]
First, by Lemma \ref{f-lem4.3}, 
for almost all $H'\in \mathcal P$, 
$\{H'\}$ satisfies $\spadesuit$. 
Next, we consider the following commutative diagram. 
$$
\xymatrix{
\widetilde X \ar[d]\ar[r]& S\times \mathbb P_{\mathbb P^1}(\mathcal E)\ar[d]\ar[r]& 
\mathbb P_{\mathbb P^1}(\mathcal E) \ar[d]\ar[r]& \mathbb P^1\\ 
X \ar[r]& S\times \mathbb P^N \ar[r]& \mathbb P^N\ar@{-->}[ur]
}
$$ 
Note that $\mathcal E=\mathscr O^{\oplus N-1}_{\mathbb P^1} \oplus 
\mathscr O_{\mathbb P^1}(1)$, 
$\mathbb P_{\mathbb P^1}(\mathcal E)\to \mathbb P^1$ is the 
blow-up along $H'_1\cap H'_2$, 
and $\mathbb P^N\dashrightarrow \mathbb P^1$ is the projection 
from $H'_1\cap H'_2$. 
In the above diagram, 
$\widetilde X$ is a resolution of the blow-up of $X$ along $h^*H'_1\cap 
h^*H'_2$ 
such that $\widetilde X$ is nothing but the blow-up of $X^\dag$ along  
$H_1^\dag\cap H_2^\dag$ over $X^\dag$ (see, 
for example, \cite{wlo}). 
We apply the argument in Step \ref{f-step1} to $\widetilde X\to 
S\times \mathbb P^1\to \mathbb P^1$ and get the desired 
property, that is, $\mathscr J(\varphi|_{H^\dag})
=\mathscr J(\varphi)|_{H^\dag}$ outside 
$H_1^\dag\cap H_2^\dag$ for 
almost all $H'\in \mathcal P$. 
Note that a point of $\mathbb P^1$ corresponds to a 
hyperplane of $\mathbb P^N$ containing $H'_1\cap H'_2$ by the projection 
$\mathbb P^N\dashrightarrow \mathbb P^1$. 
\end{proof}
\end{step}
\begin{step}\label{f-step3}
In this step, we will prove the following lemma, which is the most difficult 
part of the proof of Theorem \ref{f-thm1.1}. 

\begin{lem}\label{f-lem4.6}
There exists some $H'\in \mathcal G$ such that 
$\{H'\}$ satisfies $\spadesuit$, equivalently, $H'\in \mathcal H$. 
\end{lem}
\begin{proof}[Proof of Lemma \ref{f-lem4.6}]
If $N=1$, then this lemma 
follows from Step \ref{f-step1}. 
From now on, we assume that $N\geq 2$. 
We take two general hyperplanes $H'_1$ and $H'_2$ of $\mathbb P^N$. 
We can choose $H'_1$ and $H'_2$ such that 
$\{H'_1, H'_2\}$ satisfies $\spadesuit$ since $\Lambda$ is free. 
By Lemma \ref{f-lem4.5}, 
we can take a hyperplane $A_1$ of $\mathbb P^N$ such that 
$X_1=h^*A_1$ is smooth, 
$\{A_1\}$ satisfies $\spadesuit$, and the equality 
$$
\mathscr J(\varphi|_{X_1^\dag})
=\mathscr J(\varphi)|_{X_1^\dag}$$ 
holds outside $H_1^\dag\cap H_2^\dag$, 
where $X^\dag_1=X_1\cap X^\dag=(h^\dag)^*A_1$. 
More precisely, if $S^\dag$ is not compact, 
then we take a strictly larger open subset $\widetilde S$ with $
S^\dag\Subset \widetilde S\Subset 
S$ and apply everything to $\widetilde S$ instead of $S^\dag$. 
Then we replace $S$ with $\widetilde S$. 
By this argument, we can make $X_1=h^*A_1$ 
smooth on $X$ (not on $X^\dag$). 
By applying the induction hypothesis to $\Lambda|_{A_1}$, we see that 
$$
\{H'\in \Lambda \, |\, X_1\cap H^\dag \ 
\text{is smooth and $\mathscr J(\varphi|_{X_1\cap 
H^\dag})=\mathscr J(\varphi|_{X_1})|_{X_1\cap H^\dag}$ holds}\}
$$ 
is dense in $\Lambda$ in the classical topology, where $H^\dag=(h^\dag)^*H'$. 

We can take general hyperplanes $A_2, \ldots, A_N$ of $\mathbb P^N$ such that 
$Q=A_1\cap \cdots \cap A_N$, 
$X_Q^\dag:=X_Q\cap X^\dag$ is smooth, where 
$X_Q=h^{-1}(Q)$, and 
the equality 
$$
\mathscr J(\varphi|_{X_Q^\dag})=\mathscr J(\varphi|_{X_1})
|_{X_Q^\dag}
$$ holds 
by using the induction hypothesis repeatedly. 
As we saw above, if necessary, we apply everything 
to a strictly larger open subset $\widetilde S$ 
instead of $S^\dag$ with $S^\dag
\Subset \widetilde S\Subset S$ and replace $S$ with 
$\widetilde S$ in each step. 
Without loss of generality, we may assume that 
$X_Q^\dag\cap H_1^\dag\cap H_2^\dag=\emptyset$. 
Since $\mathscr J(\varphi|_{X_1^\dag})=\mathscr J(\varphi)|_{X_1^\dag}$ 
outside $H_1^\dag\cap H_2^\dag$, 
$$
\mathscr J(\varphi|_{X_1})|_{X_Q^\dag}=
\mathscr J(\varphi|_{X_1^\dag})|_{X_Q^\dag}
=\mathscr J(\varphi)|_{X_Q^\dag}
$$
holds. Therefore, we obtain 
$$
\mathscr J(\varphi|_{X_Q^\dag})=\mathscr J(\varphi|_{X_1})
|_{X_Q^\dag}=\mathscr J(\varphi)|_{X_Q^\dag}. 
$$ 
Of course, we can choose $A_2, A_3, \ldots, A_N$ such that 
$$
\{A_1, A_2, 
\ldots, A_N\}
$$
satisfies $\spadesuit$ with the aid of 
Lemma \ref{f-lem3.6} (see also 
Remark \ref{f-rem3.7}) since $\Lambda$ is a free linear system. 
We put 
$$
\Lambda_0=\{A\, |\, Q\in A\in |\mathscr O_{\mathbb P^N}(1)|
\}\subset \Lambda, 
$$ equivalently, $\Lambda_0$ is the sublinear system 
of $\Lambda$ spanned by 
$\{A_1, \ldots, A_N\}$.  
Then $$
\mathcal F_0=\{H'\in \Lambda_0 \, |\, \{H', A_2, \ldots, A_N\} 
\ \text{satisfies} \ \spadesuit\}
$$ 
is non-empty by $A_1\in \mathcal F_0$ 
and $\Lambda_0\setminus \mathcal F_0$ is analytically 
meagre by Lemma \ref{f-lem4.3}. 
Thus, by Lemma \ref{f-lem4.3}, we have: 
\begin{claim}\label{f-claim} 
The equality $\mathscr J(\varphi|_{X_g^\dag})=\mathscr J(\varphi)
|_{X_g^\dag}$ holds in a neighborhood 
of $X_Q^\dag$ for every $A_g\in \mathcal F_0$, where 
$X_g:=h^*A_g$. 
\end{claim}
Let $\pi:\widetilde X\to X$ be a proper bimeromorphic 
morphism from a complex manifold $\widetilde X$ 
such that $\pi:\widetilde X\to X$ is nothing but the blow-up 
of $X^\dag$ along $X_Q^\dag$ over $X^\dag$ 
(see, for example, \cite{wlo}). 
Then we have the following commutative diagram. 
$$
\xymatrix{
\widetilde X\ar[r]\ar[d]_-\pi & S\times \mathbb P(\mathcal E) 
\ar[r]\ar[d]& \mathbb P(\mathcal E)\ar[r] \ar[d]
& \mathbb P^{N-1} \\ 
X \ar[r]& S\times \mathbb P^N \ar[r]& \mathbb P^N\ar@{-->}[ur]_-{p_Q} &
}
$$
Of course, $p_Q:\mathbb P^N\dashrightarrow \mathbb P^{N-1}$ is 
the linear projection 
from $Q$ and $\mathbb P(\mathcal E)$ is the 
blow-up of $\mathbb P^N$ at $Q$, where 
$\mathbb P(\mathcal E)
=\mathbb P_{\mathbb P^{N-1}}(\mathscr O_{\mathbb P^{N-1}}
\oplus \mathscr O_{\mathbb P^{N-1}}(1))$. 
We consider the following commutative diagram. 
$$
\xymatrix{
\widetilde X\ar[drr]^-{\widetilde h}\ar[dd]_-{\widetilde f}\ar[dr]_-{\widetilde{g}} &\\ 
&S\times \mathbb P^{N-1} \ar[dl]^-{p_1}\ar[r]_-{p_2}& \mathbb P^{N-1}\\
S&
}
$$
We put $\widetilde X^\dag=\widetilde f^{-1}(S^\dag)$. 
By induction on $N$, we can take a general hyperplane 
$B$ of $\mathbb P^{N-1}$ such that 
$\widetilde h^*B\cap \widetilde X^\dag$ is smooth and 
that 
\begin{equation}\label{f-eq41}
\mathscr J(\pi^*\varphi|_{\widetilde h^*B\cap \widetilde X^\dag})
=\mathscr J(\pi^*\varphi)|_{\widetilde h^*B\cap 
\widetilde X^\dag}
\end{equation}
holds. 
Let $H'$ be the hyperplane of $\mathbb P^N$ spanned 
by $Q$ and $B$. 
Note that, by induction on $N$,  we can choose $B$ such that 
$$
\{A_2, \ldots, A_N, H'\}
$$ 
satisfies $\spadesuit$ since $\Lambda_0\setminus 
\mathcal F_0$ is analytically meagre. 
Therefore, we obtain that the equality 
$$
\mathscr J(\varphi|_{H^\dag})=\mathscr J(\varphi)|_{H^\dag} 
$$ 
holds by Claim and \eqref{f-eq41}, where $H^\dag=(h^\dag)^*H'$ as 
usual. 
More precisely, \eqref{f-eq41} implies that the equality 
$$
\mathscr J(\varphi|_{H^\dag})=\mathscr J(\varphi)|_{H^\dag}
$$ 
holds outside $X_Q^\dag$ and 
Claim implies that the equality 
$$
\mathscr J(\varphi|_{H^\dag})=\mathscr J(\varphi)|_{H^\dag}
$$ 
holds in a neighborhood of $X_Q^\dag$. 
Anyway, this $H'$ is a desired divisor. 
\end{proof}
\end{step}
\begin{step}\label{f-step4}
In this step, we will see that 
$\mathcal G\setminus \mathcal S$ 
is dense in $\Lambda$ in the classical topology 
for any analytically meagre subset $\mathcal S$ of $\Lambda$. 

By Step \ref{f-step1}, 
we may assume that $N\geq 2$. 
By Lemma \ref{f-lem4.6}, 
we can take a member $H'_0\in \mathcal G$ such that 
$\{H'_0\}$ satisfies $\spadesuit$. 
By the same argument as before, if $S^\dag$ is not compact, 
then we take a strictly larger open subset $\widetilde S$ with 
$S^\dag\Subset \widetilde S\Subset S$. 
Then we apply everything to $\widetilde S$ instead of $S^\dag$. 
By replacing $S$ with $\widetilde S$, 
we may assume that $h^*H'_0$ is smooth on $X$. 
By applying the induction hypothesis to $\Lambda|_{H'_0}$, 
we see that 
$$
\mathcal G':=\{ H'\in \Lambda \, |\, H_0^\dag\cap H^\dag\ \text{is smooth 
and $\mathscr J(\varphi|_{H_0^\dag\cap H^\dag})
=\mathscr J(\varphi|_{H_0^\dag})|_{H_0^\dag\cap H^\dag}$ holds}\}
$$ 
is dense in $\Lambda$ in the classical topology, 
where $H_0^\dag=(h^\dag)^*H'_0$ and 
$H^\dag=(h^\dag)^*H'$ as usual. 
Since $\Lambda$ is free, 
$$
\mathcal F:=\{ H'\in \Lambda \, |\, \{H'_0, H'\} \ \text{satisfies $\spadesuit$}\}
$$ 
is non-empty and $\Lambda\setminus \mathcal F$ is analytically meagre. 
Therefore, we see that  
$$
\mathcal G'':=\{ H'\in \mathcal G' \, |\, \{H'_0, H'\} \ 
\text{satisfies $\spadesuit$}\}
$$ 
is also dense in $\Lambda$ in the classical topology 
with the aid of Lemma \ref{f-lem3.6}. 
We note that 
\begin{equation}\label{f-eq42}
\mathscr J(\varphi|_{H_0^\dag\cap H_1^\dag})=\mathscr J(\varphi |_
{H_0^\dag})|_{H_0^\dag\cap H_1^\dag}
=\mathscr J(\varphi)|_{H_0^\dag\cap H_1^\dag}
\end{equation} 
with $H_1^\dag=(h^\dag)^*H'_1$ 
for every $H'_1\in \mathcal G'$ since 
$$
\mathscr J(\varphi|_{H_0^\dag})=\mathscr J(\varphi)|_{H_0^\dag}. 
$$ 
Here, we used the fact that $H'_0\in \mathcal G$. 
We consider the pencil $\mathcal P$ spanned by 
$H'_0$ and $H'_1\in \mathcal G''$, that is, 
the sublinear system of $\Lambda$ spanned by 
$H'_0$ and $H'_1$. 
By Lemma \ref{f-lem4.5}, for almost all $H'\in \mathcal P$, $H^\dag=
(h^\dag)^*H'$ is smooth and 
the equality 
$$
\mathscr J(\varphi|_{H^\dag})=\mathscr J(\varphi)|_{H^\dag}
$$ 
holds outside $H_0^\dag\cap H^\dag=H_0^\dag\cap H_1^\dag$. 
For almost all $H'\in \mathcal P$, $\{H'_0, H'\}$ satisfies $\spadesuit$ by 
Lemma \ref{f-lem4.3}. 
Therefore, 
the equality 
$$
\mathscr J(\varphi|_{H^\dag})=\mathscr J(\varphi)|_{H^\dag}
$$
holds in a neighborhood of $H_0^\dag\cap H^\dag=
H_0^\dag\cap H_1^\dag$ for almost all $H'\in \mathcal P$ by 
Lemma \ref{f-lem4.3} and \eqref{f-eq42}. 
Therefore, for almost all $H'\in \mathcal P$, $H^\dag$ is smooth and 
the equality 
$$
\mathscr J(\varphi|_{H^\dag})=\mathscr J(\varphi)|_{H^\dag}
$$
holds. 
This means that almost 
all members of $\mathcal P$ are contained in $\mathcal G$. 

Let $P$ be a point of $\mathbb P^N\simeq \Lambda$ corresponding 
to $H'_0$. 
We put $\Sigma=\Lambda\setminus \mathcal F$, 
$\mathcal G_{N-1}=\mathcal G'|_{H'_0}\subset \Lambda|_{H'_0}$, 
and $\mathcal G_N=\mathcal G$. 
Then we can apply Lemma \ref{f-lem3.6} (see also Remark \ref{f-rem3.7}). 
Therefore, $\mathcal G\setminus \mathcal S$ is 
dense in $\Lambda$ in the classical 
topology 
for any analytically meagre subset $\mathcal S$ of $\Lambda$. 
\end{step} 

\begin{step}\label{f-step5}
In this step, we will see that $\mathcal H$ is dense in $\Lambda$ in the classical 
topology. 

We put 
$$
\left\{H'\in \Lambda \, \left|
\begin{array}{l}
\mbox{$H^\dag:=(h^\dag)^*H'$ is well-defined and contains}\\ 
\mbox{no associated primes of $\mathscr O_X/\mathscr J(\varphi)$ 
on $X^\dag$}
\end{array}
\right.
\right\}. 
$$
Then it contains a non-empty Zariski open subset of $\Lambda$. 
Note that $\Lambda$ is free and the number of 
the associated primes of $\mathscr O_X/\mathscr J(\varphi)$ on $X^\dag$ 
is finite. 
Therefore, by Step \ref{f-step4}, $\mathcal H$ is dense in $\Lambda$ in the classical 
topology because $\mathcal G\setminus \mathcal H$ is contained in 
an analytically meagre subset of $\Lambda$. 
\end{step}

\begin{step}\label{f-step6}
Let $H'$ be a member of $\mathcal H$. 
We consider the following big commutative diagram. 
\begin{equation*}
\xymatrix{
 & 0\ar[d] & 0\ar[d]&\\
0 \ar[r]& \mathscr J(\varphi|_{X^\dag}) \otimes \mathscr O_{X^\dag}(-H^\dag)
\ar[d]
\ar[r]^{\quad \quad \quad \alpha}
&\mathscr J(\varphi|_{X^\dag})\ar[d]\ar[r]& \Coker \alpha 
\ar[r]\ar[d]^{\beta}& 0 \\ 
0 \ar[r]& \mathscr O_{X^\dag}(-H^\dag)\ar[d]
\ar[r]
&\mathscr O_{X^\dag} \ar[d]\ar[r]& \mathscr O_{H^\dag}\ar[r]& 0 \\
&\left(\mathscr O_{X^\dag}/\mathscr J(\varphi|_{X^\dag})\right)\otimes 
\mathscr O_{X^\dag}(-H^\dag) \ar[r]^{\quad \quad \quad \gamma}
\ar[d]& \mathscr O_{X^\dag}/\mathscr J(\varphi|_{X^\dag})\ar[d]&\\
& 0& 0&
}
\end{equation*} 
Of course, $H^\dag=(h^\dag)^*H'$ in the above diagram. 
Since $H'\in \mathcal H$, $\gamma$ is injective. 
Therefore, $\beta$ is also injective by the snake lemma. 
Thus we obtain that $$\Coker \alpha=\mathscr J(\varphi|_{X^\dag})|_{H^\dag}$$ 
by definition. 
Then we have the following desired short exact sequence 
$$
0\to \mathscr J(\varphi|_{X^\dag})\otimes 
\mathscr O_{X^\dag}(-H^\dag)\to \mathscr J(\varphi|_{X^\dag})
\to \mathscr J(\varphi|_{H^\dag})\to 0
$$ 
because $\mathscr J(\varphi|_{H^\dag})=\mathscr J(\varphi|_{X^\dag})|_{H^\dag}$ 
holds for $H'\in \mathcal H$. 
\end{step}
We complete the proof of Theorem \ref{f-thm1.1}. 
\end{proof}

We close this section with a remark on the set $\Lambda \setminus \mathcal G$. 

\begin{rem}[cf.~Question \ref{f-question1.2}]\label{f-rem4.7}
Theorem \ref{f-thm1.1} says that $\mathcal G$ is dense in $\Lambda$ 
in the classical topology. 
However, the proof of Theorem \ref{f-thm1.1} 
gives no information on the set $\Lambda\setminus \mathcal G$ in $\Lambda$ 
($\simeq \mathbb P^N$). 
This is because we use Lemma \ref{f-lem3.6} for 
induction on $N$. 
We do not know whether 
$\Lambda\setminus \mathcal G$ has measure zero or not 
(see \cite[Examples 3.11 and 3.12]{fujino-matsumura}). 
\end{rem}

\section{Proof of Theorem \ref{f-thm1.4}}\label{f-sec5}

In this section, we prove Theorem \ref{f-thm1.4} as an 
application of Theorems \ref{f-thm1.1} and \ref{f-thm1.3}. 

\begin{proof}[Proof of Theorem \ref{f-thm1.4}]
We take an arbitrary point $z\in Z$. 
Let us prove 
\begin{equation*}
R^i\pi_*R^jf_*(\omega_X\otimes F\otimes 
\mathscr J(h))=0
\end{equation*} 
for every $i>0$ and $j$ in a neighborhood 
of $z$ by induction on $\dim \pi^{-1}(z)$. 
Without loss of generality, we may assume that 
$f_*\mathscr O_X\simeq \mathscr O_Y$ and $\pi_*\mathscr O_Y\simeq 
\mathscr O_Z$ by taking the Stein factorizations of $f$ and $\pi$. 
Since $\pi\circ f$ is locally K\"ahler (see, for example, 
\cite[Proposition 6.2 (ii)]{takegoshi}), we may assume that $X$ is 
K\"ahler by shrinking $Z$ around $z$. 
If $\dim \pi^{-1}(z)=0$, then 
$\pi:Y\to Z$ is finite over a neighborhood of $z$. 
In this case, it is obvious that $R^i\pi_*R^jf_*(\omega_X\otimes 
F\otimes \mathscr J(h))=0$ holds for every $i>0$ and $j$ in a neighborhood 
of $z$. 
From now on, we assume that $\dim \pi^{-1}(z)>0$. 
By replacing $H$ with $H^{\otimes m}$ for some sufficiently large positive 
integer $m$, 
we may assume that $H$ is $\pi$-very ample and 
\begin{equation}\label{f-eq51}
R^i\pi_*(H\otimes R^jf_*(\omega_X\otimes F \otimes \mathscr J(h)))=0 
\end{equation}
for every $i>0$ and $j$ (see, for 
example \cite[Chapter IV, Theorem 2.1 (B)]{banica-s}). 
We may further assume that there exists the following commutative diagram 
\begin{equation*}
\xymatrix{
X\ar[d]_-f \ar[dr]&&\\ 
Y\ar[d]_-\pi \ar@{^{(}->}[r]& Z\times \mathbb P^N \ar[dl]^-{p_1}
\ar[r]_-{p_2}& \mathbb P^N \\ 
Z &&
}
\end{equation*}
such that $H\simeq (p^*_2\mathscr O_{\mathbb P^N}(1))|_Y$ by 
shrinking $Z$ around $z$ suitably (see, 
for example \cite[Chapter IV, \S 2]{banica-s}). 
By Theorem \ref{f-thm1.1}, we can take 
a general member $A'$ of $|\mathscr O_{\mathbb P^N}(1)|$ such that 
$A_Y:=(p_2^*A')|_Y$ contains no associated primes 
of $R^jf_*(\omega_X\otimes 
F\otimes \mathscr J(\varphi))$ for every $j$, 
$A_Y$ contains no irreducible components of $\pi^{-1}(z)$, 
$A$ is smooth, where $A=f^*A_Y$, and 
$$
0\to \mathscr J(h)\otimes \mathcal O_X(-A)\to 
\mathscr J(h)\to \mathscr J(h|_{A})\to 0
$$ 
is exact after shrinking $Z$ around $z$ suitably. 
Therefore, by adjunction, 
$$
0\to \omega_X\otimes F\otimes \mathscr J (h)\to 
\omega_X\otimes F\otimes \mathscr J(h)\otimes \mathcal O_X(A)
\to \omega_{A}\otimes F|_{A}\otimes \mathscr J(h|_{A})\to 0
$$
is exact. 
Thus, we see that 
\begin{equation*}
\begin{split}
0&\to R^jf_*(\omega_X\otimes F\otimes \mathscr J(h)) 
\to R^jf_*(\omega_X\otimes F\otimes \mathscr J(h)\otimes \mathscr O_X(A))
\\
&\to R^jf_*(\omega_{A}\otimes F|_{A}\otimes \mathscr J(h|_{A}))
\to 0 
\end{split}
\end{equation*}
is exact for every $j$ since $A_Y$ contains no associated 
primes of $R^jf_*(\omega_X\otimes F\otimes \mathscr J(h))$ for 
every $j$. 
We note that $R^jf_*(\omega _{A} \otimes F|_{A} 
\otimes \mathscr J(h|_{A}))$ is $\pi_*$-acyclic in a neighborhood 
of $z$ by induction on $\dim \pi^{-1}(z)$ and 
that 
$R^jf_*(\omega _X\otimes F\otimes \mathscr J(h)\otimes 
\mathscr O_X(A))$ is $\pi_*$-acyclic by the above assumption 
(see \eqref{f-eq51}). 
We consider 
the long exact sequence: 
\begin{equation*}
\begin{split}
\cdots&\to R^i\pi_*R^jf_*(\omega_X\otimes F\otimes \mathscr J(h)) 
\to R^i\pi_*R^jf_*(\omega_X
\otimes F\otimes \mathscr J(h)\otimes \mathscr O_X(A))
\\
&\to R^i\pi_*R^jf_*(\omega_{A}\otimes F|_{A}\otimes \mathscr J(h|_{A}))
\to \cdots.  
\end{split}
\end{equation*}
Thus, if we shrink $Z$ around $z$ suitably, then 
we have $E^{i, j}_2=0$ for every $i\geq 2$ and $j$ in the 
following commutative 
diagram of spectral sequences. 
\begin{equation*}
\xymatrix{
E^{i, j}_2=R^i\pi_*R^jf_*(\omega_X\otimes F\otimes 
\mathscr J(h)) \ar@{=>}[r]\ar[d]_-{\varphi^{i,j}} 
& R^{i+j} (\pi\circ f)_*(\omega_X
\otimes F\otimes \mathscr J(h))\ar[d]^-{\varphi^{i+j}} \\ 
\overline{E}^{i, j}_2=R^i\pi_*R^jf_*(\omega_X\otimes F\otimes 
\mathscr J(h)\otimes 
\mathscr O_X(A)) \ar@{=>}[r] & R^{i+j} (\pi\circ f)_*(\omega_X
\otimes F\otimes \mathscr J(h)\otimes \mathscr O_X(A))
}
\end{equation*} 
We note that $\varphi^{i+j}$ is injective by Theorem \ref{f-thm1.3}. 
We also note that 
\begin{equation*}
\xymatrix{
E^{1, j}_2\ar[r]^-{\alpha} & R^{1+j} (\pi\circ 
f)_*(\omega_X\otimes F\otimes \mathscr J(h))
}
\end{equation*} 
is injective for every $j$ by the fact that $E^{i,j}_2=0$ for 
every $i\geq 2$ and $j$.  
By the above assumption (see \eqref{f-eq51}), we have  
$\overline {E}^{1,j}_2=0$ for every $j$. 
Therefore, we obtain $E^{1,j}_2=0$ for every $j$ since 
the injection 
\begin{equation*}
\xymatrix{
E^{1,j}_2\ar[r]^-{\alpha} &
R^{1+j}(\pi\circ f)_*(\omega_X\otimes F\otimes \mathscr J(h)) 
\ar[r]^-{\varphi^{1+j}}&
R^{1+j} (\pi\circ f)_* (\omega _X\otimes F\otimes 
\mathscr J(h)\otimes \mathscr O_X(A)) 
}
\end{equation*} 
factors through $\overline {E}^{1, j}_2=0$. 
This implies that 
$R^i\pi_*R^jf_*(\omega_X\otimes F\otimes \mathscr J(h))=0$ for 
every $i>0$ and $j$ in a neighborhood 
of an arbitrary point $z\in Z$. 
This means that 
\begin{equation*}
R^i\pi_*R^jf_*(\omega_X\otimes F\otimes \mathscr J(h))=0
\end{equation*} 
for 
every $i>0$ and $j$.
\end{proof} 

\section{Proof of Corollary \ref{f-cor1.5}}\label{f-sec6}

By using the strong openness in \cite{guan-zhou}, we can prove 
Corollary \ref{f-cor1.5} as an easy application of 
Theorem \ref{f-thm1.4}. 

Let us prepare a lemma suitable for our application. 

\begin{lem}[{cf.~\cite[Theorem 1.1]{guan-zhou}}]\label{f-lem6.1}
Let $X$ be a complex manifold and let $\varphi$ and $\psi$ be 
quasi-plurisubharmonic functions on $X$. 
Let $X^\dag$ be a relatively compact open subset 
of $X$. 
Then there exists a small positive number $\varepsilon$ such that 
$$
\mathscr J(\varphi)=\mathscr J(\varphi+\varepsilon \psi)
$$ 
holds on $X^\dag$. 
\end{lem}
\begin{proof}
By definition, it is obvious that the natural inclusion  
$$
\mathscr J(\varphi)\supset \mathscr J(\varphi+\varepsilon \psi)
$$ 
holds since $\varepsilon$ is positive. 

Let us see the problem locally. 
Let $\Delta^n=\{(z_1, \ldots, z_n)\, |\, |z_1|<1, \cdots, |z_n|<1\}$ 
be the unit polydisc and let $\varphi$ and $\psi$ be plurisubharmonic 
functions on $\Delta^n$. 
Let $f_1, \ldots, f_k$ be holomorphic functions 
on $\Delta^n$ such that 
$$
\int _{\Delta^n} |f_i|^2e^{-2\varphi} d\lambda_n<\infty
$$ 
for every $i$, where $d\lambda_n$ is the Lebesgue measure 
on $\mathbb C^n$, and that $\{f_1, \ldots, f_k\}$ 
generates $\mathscr J(\varphi)_0$, the stalk of $\mathscr J(\varphi)$ at $0\in \Delta^n$. 
By \cite[Theorem 1.1]{guan-zhou}, 
we can take $r\in (0, 1)$ and $p>1$ such that 
$$
\int _{\Delta^n_r}|f_i|^2e^{-2p\varphi}d\lambda_n<\infty
$$ 
for every $i$, where $\Delta^n_r=\{(z_1, \ldots, z_n) \, |\, 
|z_1|<r, \cdots, |z_n|<r\}$. 
We put $q=\frac{p}{p-1}>0$. 
Then, by the H\"older inequality, we have 
$$
\int _{\Delta^n_r}|f_i|^2e^{-2(\varphi+\varepsilon \psi)}d\lambda_n 
\leq 
\left(\int _{\Delta^n_r}|f_i|^2e^{-2p\varphi}d\lambda_n\right)^{1/p}
\left(\int _{\Delta^n_r}|f_i|^2e^{-2q\varepsilon \psi}d\lambda_n\right)^{1/q}. 
$$
By replacing $r$ with a smaller positive number, 
we can take  
$\varepsilon>0$ such that 
$$
\int_{\Delta^n_r} e^{-2q\varepsilon \psi}d\lambda_n<\infty
$$
by Skoda's theorem (see, for example, \cite[(5.6) Lemma]{demailly}). 
Then we obtain 
$$
\int _{\Delta^n_r}|f_i|^2e^{-2(\varphi+\varepsilon \psi)}d\lambda_n<\infty. 
$$
This implies that $f_i\in \mathscr J(\varphi+\varepsilon \psi)_0$ for 
every $i$. 
Therefore, we obtain the inclusion 
$$
\mathscr J(\varphi)_0\subset \mathscr J(\varphi+\varepsilon \psi)_0.  
$$ 
Then the equality 
$$
\mathscr J(\varphi)_0=\mathscr J(\varphi+\varepsilon \psi)_0
$$ 
holds. So, the equality 
$$
\mathscr J(\varphi)=\mathscr J(\varphi+\varepsilon \psi)
$$ 
holds in a neighborhood of $0\in \Delta^n$ since 
$\mathscr J(\varphi)$ and $\mathscr J(\varphi+\varepsilon \psi)$ are both 
coherent. 

Thus, we can take $\varepsilon >0$ such that 
$$
\mathscr J(\varphi)=\mathscr J(\varphi+\varepsilon \psi)
$$ 
holds on $X^\dag$ since $X^\dag$ is a relatively compact open subset of $X$. 
\end{proof}

Let us prove Corollary \ref{f-cor1.5}. 

\begin{proof}[Proof of Corollary \ref{f-cor1.5}]
It is sufficient to prove that 
$$
R^i\pi_*(M\otimes R^jf_*(\omega_X\otimes F\otimes \mathscr J(h)))=0
$$ 
holds for every $i>0$ and $j$ in a neighborhood of 
any fixed point $z\in Z$. 
By shrinking $Z$ around $z$, we may assume that 
$X$ is K\"ahler since $\pi\circ f$ is locally K\"ahler (see, 
for example, \cite[Proposition 6.2 (ii)]{takegoshi}). 
Without loss of generality, we may assume that 
$Z$ is Stein. 
By shrinking $Z$ around $z$, we may further assume that 
there exists the following commutative diagram 
since $\pi:Y\to Z$ is locally projective. 
$$
\xymatrix{
X \ar[drr]^-h\ar[d]_-f\ar[dr]& & \\ 
Y \ar[d]_-\pi\ar@{^{(}->}[r]& Z\times \mathbb P^N
\ar[dl]^-{p_1}\ar[r]_-{p_2} 
& \mathbb P^N\\ 
Z & & 
}
$$
Then we can take a sufficiently large and divisible positive integer 
$m$ such that 
$$
M^{\otimes m} \simeq H\otimes \mathscr O_Y(E)
$$
where $H\simeq (p^*_2\mathscr O_{\mathbb P^N}(1))|_Y$ and 
$E$ is an effective Cartier divisor on $Y$ by Kodaira's lemma. 
Then we obtain 
$$
M^{\otimes (2m+k)}\simeq (M^{\otimes k}\otimes H)\otimes 
H\otimes \mathscr O_Y(2E). 
$$ 
We note that $M^{\otimes k}\otimes H$ is $\pi$-ample 
for every positive integer $k$ since $M$ is $\pi$-nef.
Since $f^*H\simeq 
h^*\mathscr O_{\mathbb P^N}(1)$, we can construct 
a smooth hermitian metric $g$ on $f^*H$ such that 
$\sqrt{-1}\Theta_g(f^*H)\geq 0$. 
Similarly, $f^*(M^{\otimes k}\otimes H)$ has a smooth 
hermitian metric $g_1$ such that 
$\sqrt{-1}\Theta_{g_1}(f^*(M^{\otimes k}\otimes H))\geq 0$ after 
shrinking $Z$ around $z$ suitably because 
$M^{\otimes k}\otimes H$ is $\pi$-ample. 
Let $s$ be the canonical section of $\mathscr O_X(f^*E)$, that is, 
$s\in \Gamma (X, \mathscr O_X(f^*E))$ with 
$(s=0)=f^*E$. 
Let $g_2$ be any smooth hermitian metric on $\mathscr O_X(f^*E)$. 
We put 
$$
g_3=\frac{g_2}{|s|^2_{g_2}}. 
$$ 
Then $g_3$ is a singular hermitian metric on $\mathscr O_X(f^*E)$ such that 
$\sqrt{-1} \Theta_{g_3}(\mathscr O_X(f^*E))\geq 0$ and that 
$g_3$ is smooth outside $\Supp f^*E$. 
We put 
$$
h'=(g_1\cdot g\cdot g^2_3)^{\frac{1}{2m+k}}. 
$$ 
Then $h'$ is a singular hermitian metric on $f^*M$, which is smooth 
outside $\Supp f^*E$. 
By construction, 
\begin{equation*} 
\sqrt{-1} \Theta_g(f^*H)\geq 0 \quad \text{and} \quad 
\sqrt{-1}\Theta_{h'}(f^*M)-\varepsilon \sqrt{-1}\Theta _g (f^*H)\geq 0
\end{equation*} 
for some $\varepsilon >0$. 
If $k$ is sufficiently large, then we can make 
$h'$ satisfy $\mathscr J(hh')=\mathscr J(h)$ by 
Lemma \ref{f-lem6.1}. 
We note that we can freely shrink $Z$ around $z$ if necessary. 
Anyway, this means that 
\begin{equation*} 
\sqrt{-1} \Theta_g(f^*H)\geq 0 \quad \text{and} \quad
\sqrt{-1}\Theta_{hh'}(F\otimes 
f^*M)-\varepsilon \sqrt{-1}\Theta _g (f^*H)\geq 0
\end{equation*} 
for some $\varepsilon>0$ such that the equality $\mathscr J(hh')
=\mathscr J(h)$ holds. 
By applying Theorem \ref{f-thm1.4}, 
we obtain that 
\begin{equation*}
R^i\pi_*(M\otimes R^jf_*(\omega_X\otimes F\otimes \mathscr J(h)))
=R^i\pi_*R^jf_*(\omega _X\otimes F\otimes f^*M\otimes 
\mathscr J(hh'))=0 
\end{equation*}
holds for every $i>0$ and $j$. 
\end{proof}



\begin{thebibliography}{99}

\bibitem{banica-s}
C.~B\u anic\u a, O.~St\u an\u a\c sil\u a, 
{\em{Algebraic methods in the global theory of complex spaces}}, 
Translated from the Romanian. Editura Academiei, Bucharest; 
John $\&$ Wiley  Sons, London-New York-Sydney, 1976.

\bibitem{demailly-book} J.-P.~Demailly, 
{\em{Complex analytic and differential geometry}}, 
available at the author's homepage. 

\bibitem{demailly} J.-P.~Demailly, 
{\em{Analytic methods in algebraic geometry}}, Surveys 
of Modern Mathematics, \textbf{1}. International 
Press, Somerville, MA; Higher Education Press, Beijing, 2012. 

\bibitem{enoki}
I.~Enoki, Kawamata--Viehweg vanishing theorem for 
compact K\"ahler manifolds, {\em{Einstein metrics and 
Yang--Mills connections (Sanda, 1990)}}, 59--68, 
Lecture Notes in Pure and Appl. Math., {\textbf{145}}, 
Dekker, New York, 1993. 

\bibitem{fischer}
G.~Fischer, 
{\em{Complex analytic geometry}}, 
Lecture Notes in Mathematics, \textbf{538}. Springer-Verlag, 
Berlin-New York, 1976. 

\bibitem{fujino1}
O.~Fujino, A transcendental approach to 
Koll\'ar's injectivity theorem II, J. Reine 
Angew. Math. \textbf{681} (2013), 149--174.

\bibitem{fujino2} 
O.~Fujino, Koll\'ar--Nadel type vanishing theorem, to 
appear in Southeast Asian Bull. Math. 

\bibitem{fujino3} 
O.~Fujino, {\em{Foundations of the minimal model program}}, 
MSJ Memoirs, \textbf{35}. Mathematical Society of Japan, Tokyo, 2017. 

\bibitem{fujino-matsumura} 
O.~Fujino, S.~Matsumura, 
Injectivity theorem for pseudo-effective line bundles and its applications, 
preprint (2016). arXiv:1605.02284 [math.CV]

\bibitem{guan-zhou}
Q.~Guan, X.~Zhou, 
A proof of Demailly's strong openness conjecture, 
Ann. of Math. (2) \textbf{182} (2015), no. 2, 605--616. 

\bibitem{kollar}
J.~Koll\'ar, 
Higher direct images of dualizing sheaves. I, 
Ann. of Math. (2) {\textbf{123}} (1986), no. 1, 11--42.

\bibitem{manaresi} 
M.~Manaresi, Sard and Bertini type theorems for complex spaces, 
Ann. Mat. Pura Appl. (4) \textbf{131} (1982), 265--279. 

\bibitem{matsumura1} 
S.~Matsumura, A vanishing theorem of Koll\'ar--Ohsawa type, 
Math. Ann. \textbf{366} (2016), no. 3-4, 1451--1465

\bibitem{matsumura2}
S.~Matsumura, 
Injectivity theorems with multiplier ideal sheaves for higher direct images 
under K\"ahler morphisms, preprint (2016). arXiv:1607.05554 [math.CV]

\bibitem{ohsawa} 
T.~Ohsawa, Vanishing theorems on complete K\"ahler manifolds, 
Publ. Res. Inst. Math. Sci. 20 (1984), no. 1, 21--38. 

\bibitem{takegoshi} 
K.~Takegoshi, Higher direct images of canonical sheaves tensorized with 
semi-positive vector bundles by proper K\"ahler morphisms, 
Math. Ann. \textbf{303} (1995), no. 3, 389--416. 

\bibitem{wlo} J.~W\l odarczyk, 
Resolution of singularities of analytic spaces, 
Proceedings of G\"okova Geometry-Topology Conference 2008, 31--63, 
G\"okova Geometry/Topology Conference (GGT), G\"okova, 2009.
\end{thebibliography}
\end{document}